\documentclass[a4paper,11pt,twoside]{amsart}

\usepackage{srollens-en}[2011/03/03]
\usepackage[left=35mm,right=35mm,headsep=6mm,footskip=10mm,top=36mm,bottom=40mm,footnotesep=5mm,headheight=20mm]{geometry}
\usepackage{subfig}

\newlist{prooflist}{description}{1}
\setlist[prooflist]{font=\normalfont \itshape, labelindent = \parindent, leftmargin = 0pt}

\makeatletter
\g@addto@macro\@floatboxreset\centering
\makeatother

\usepackage[unicode,bookmarks,pdftex]{hyperref}
\hypersetup{colorlinks=true,citecolor=NavyBlue,linkcolor=NavyBlue,urlcolor=Orange,pdfpagemode=UseNone,breaklinks=true}

\renewcommand{\lg}{\mathfrak g} 
\newcommand{\lgbar}{\overline\lg}
\newcommand{\lf}{\mathfrak f}
\newcommand{\lh}{\mathfrak h}
\newcommand{\frn}{\mathfrak{n}}

\DeclareMathOperator{\Kur}{Kur}
\DeclareMathOperator{\obs}{obs}
\DeclareMathOperator{\Gl}{GL}
\DeclareMathOperator{\ad}{ad}
\DeclareMathOperator{\KS}{KS}

\title{Verbal ideals and unobstructed complex parallelisable nilmanifolds}

\author{Matthias Paulsen}
\address{Matthias Paulsen\\FB 12/Mathematik und Informatik\\
Philipps-Universit\"at Marburg\\
Hans-Meerwein-Str. 6\\
35032 Marburg\\
Germany}
\email{paulsen@mathematik.uni-marburg.de}

\author{S\"onke Rollenske}
\address{S\"onke Rollenske\\FB 12/Mathematik und Informatik\\
Philipps-Universit\"at Marburg\\
Hans-Meerwein-Str. 6\\
35032 Marburg\\
Germany}
\email{rollenske@mathematik.uni-marburg.de}

\author{Konstantin Wehler}
\address{Konstantin Wehler\\FB 12/Mathematik und Informatik\\
Philipps-Universit\"at Marburg\\
Hans-Meerwein-Str. 6\\
35032 Marburg\\
Germany}
\email{konstantin.wehler@uni-marburg.de}

\begin{document}
\begin{abstract}
We show that a compact complex parallelisable nilmanifold has unobstructed deformations if and only if its associated Lie algebra satisfies a reality condition and is a free Lie algebra in a variety of Lie algebras, that is, defined by a verbal ideal in a free Lie algebra.

We provide a  partial classification of verbal ideals and show that there are finitely many such Lie algebras up to dimension~$19$, whereas infinite families start to appear in dimension~$20$.

As a consequence, there are finitely many complex homotopy types of unobstructed complex parallelisable nilmanifolds up to dimension~$19$, and infinitely many in dimension~$20$.
\end{abstract}

\maketitle

\section{Introduction}
The study of deformations of complex structures on a compact complex manifold $X$ has been an important tool ever since it was first developed by Kodaira and Spencer in \cite{Kod-sp58}. While Kuranishi proved in \cite{kuranishi62} the existence of a complex space $\Kur(X)$, nowadays known as the Kuranishi space, parametrising all sufficiently small deformations of $X$ in the most effective way, the actual computation of deformations is usually intractable without further assumptions.

If the canonical bundle of $X$ is trivial and $X$ admits a K\"ahler metric, then deformations are unobstructed, that is, $\Kur(X)$ is smooth by the famous Tian--Todorov Lemma, but in general the Kuranishi space can be arbitrarily singular \cite{Vakil}.

In \cite{rollenske2011kuranishi} the second author showed that on compact complex parallelisable nilmanifolds, the computation of the Kuranishi space can be performed in finitely many steps from the structure equations of the associated nilpotent complex Lie algebra. One deduces that in many cases the deformation space is singular, showing that there does not exist a Tian--Todorov type result in this context.
But the converse question, which complex parallelisable nilmanifolds have unobstructed deformations,  was only addressed up to second order in \cite{rollenske2011kuranishi} hinting at some kind of freeness property of the Lie algebra.

In this paper, we present a precise characterisation of complex parallelisable nilmanifolds with unobstructed deformations.
Our results draw upon the theory of identities in Lie algebras \cite{bahturin2021identical}:
let $R$ be a set of non-associative polynomials.
The class of all Lie algebras satisfying the relations in $R$ form a so-called variety  of Lie algebras $\IV$.
If a Lie algebra $\lg\in\IV$ satisfies no other relations than the ones required by $\IV$,
then $\lg$ is called a free Lie algebra of $\IV$.
In order to suppress the specific variety $\IV$ from the notation, we call such a Lie algebra \emph{pseudo-free}; in particular, freely nilpotent Lie algebras are pseudo-free in this sense.

We can now state our first main result, to be proved in Section~\ref{sect: geometry}.
\begin{thm}\label{thm: main deformation}
Let $X$ be a compact complex parallelisable nilmanifold with associated Lie algebra $\lg$. Then $X$ has unobstructed deformations if and only if $\lg$ is pseudo-free and $\lg \cong \overline\lg$ as complex Lie algebras.
\end{thm}
In terms of an appropriately chosen basis, the condition $\lg \cong \overline\lg$ can be interpreted as a reality condition on the structure constants of $\lg$.

It is easy to see (cf.\ Lemma~\ref{lem: pseudo-free}) that a pseudo-free Lie algebra is precisely the quotient of a free Lie algebra by a so-called \emph{verbal ideal}, that is, an ideal invariant under all endomorphisms of the free Lie algebra. Therefore, Theorem~\ref{thm: main deformation} connects the classification of complex parallelisable nilmanifolds with unobstructed deformations to the classification of verbal ideals in free Lie algebras. 

In Section~\ref{sect: verbal}, we explain how standard tools from representation theory can be used to determine the possible homogeneous parts of a verbal ideal. Combined with results of Zhuravlev \cite{zhuravlev1997}, we obtain a complete classification for low nilpotency index or small dimension: let us denote by $\gothn_{m,\nu}$ the free $\nu$-step nilpotent Lie algebra on $m$ generators.

\begin{thm}\label{thm: pseudofree classification}
 Let $\lg$ be a non-abelian pseudo-free $\nu$-step nilpotent complex Lie algebra with $\dim (\lg/[\lg,\lg] )= m$.
 \begin{enumerate}
 \item 
 If $\nu \leq 3$, then  $\lg\isom \frn_{m,\nu}$ is a freely nilpotent Lie algebra.
 \item If $\nu \leq 5$, then $\lg$ is isomorphic to one of the finitely many Lie algebras listed in Table~\ref{tab: nilpotency index 5}.
 \item If $\dim \lg\leq 20$, then $\lg$ is isomorphic to one of the Lie algebras listed in Table~\ref{tab: small dim}, which contains $19$ individual Lie algebras and one $1$-parameter family of $20$-dimensional Lie algebras.
 \end{enumerate}
 \end{thm}
 	\begin{table}
	\caption{Non-abelian pseudo-free Lie algebras up to dimension $20$ (see Section~\ref{sect: verbal} for notation)}
	\label{tab: small dim}
	\begin{tabular}{clr}
        \toprule
		Dimension & Lie algebras & $\nu$ \\
		\midrule
		$\phantom{1}3$ & $\frn_{2,2}$& $2$\\
		$\phantom{1}5$ & $\frn_{2,3}$& $3$\\
		$\phantom{1}6$ & $\frn_{3,2}$& $2$\\
		$\phantom{1}8$ & $\frn_{2,4}$& $4$\\
		\midrule
		$10$ & $\frn_{4,2}$,& $2$\\
		& $\frn_{2,5}/V_{(4,1)}$& $5$\\
		\midrule
		$11$ & $\frn_{2,6}/\left( V_{(4,1)}\oplus V_{(5,1)}\oplus V_{(4,2)}\right)$& $6$\\
		$12$ & $\frn_{2,5}/V_{(3,2)}$& $5$\\
		\midrule
		$14$ & $\frn_{2,5}$,& $5$\\ 
		& $\frn_{3,3}$ & $3$\\
		\midrule
		$15$ & $\frn_{5,2}$, & $2$\\
		& $\frn_{2,6}/\left(V_{(5,1)}\oplus V_{(4,2)}\right)$& $6$\\
		\midrule
		$17$ & $\frn_{3,4}/V_{(3,1)} $,& $4$\\
		& $\frn_{2,6}/\left( V_{(5,1)}\oplus V_{(3,3)}\right)$,& $6$\\
		& $\frn_{2,6}/\left( V_{(3,2)}\oplus V_{(4,2)}\oplus V_{(3,3)}\right)$,& $6$\\
		& $\frn_{2,7}/\mathfrak{a}$ (Example~\ref{exam: 7-step, 2gen}) & $7$ \\
				\midrule
		$18$ & $\frn_{2,6}/V_{(5,1)}$& $6$\\
		$19$ & $\frn_{2,6}/\left(V_{(4,2)}\oplus V_{(3,3)}\right)$& $6$\\
		\midrule
		$20$ & $\frn_{2,6}/V_{(4,2)}$, & $6$\\
		& $\lg_\mu = \frn_{2,7}/\mathfrak a_\mu$ with $\mu \in \IP^1_\IC$ (Example~\ref{exam: smallest})& $7$ \\
		\bottomrule
	\end{tabular}
\end{table}

The classification could be extended to higher dimensions by the same methods, but Examples~\ref{exam: 6step, 3gen} and \ref{exam: 7 step always infinite} show that for any pair $(m, \nu)$ not covered by Theorem~\ref{thm: pseudofree classification} there exist infinite families of non-isomorphic nilpotent pseudo-free Lie algebras.

Translating back to geometry and adding in the reality condition and the condition for the existence of a lattice in the nilpotent Lie group, we obtain the following classification result on complex parallelisable nilmanifolds with unobstructed deformations:
\begin{cor}\label{cor: classification cx parall}
Let $X$ be a compact complex parallelisable nilmanifold which is not a complex torus and let $\lg$ be the associated Lie algebra of nilpotency index~$\nu$.
 \begin{enumerate}
  \item\label{it:cx.nilp3} If $\nu\leq 3$, then $X$ has unobstructed deformations if and only if $\lg$ is a freely nilpotent Lie algebra.
  \item\label{it:cx.nilp5} If $\nu\leq 5$, then $X$ has unobstructed deformations if and only if $\lg$ is one of the Lie algebras in Table~\ref{tab: nilpotency index 5}, and all of these can occur.
  \item\label{it:cx.dim19} If $\dim X\leq 19$, then $X$ has unobstructed deformations if and only if $\lg$ is one of the finitely many Lie algebras in Table~\ref{tab: small dim}, and all of these can occur.
  \item\label{it:cx.dim20} In dimension~$20$ there are infinitely many different complex homotopy types of complex parallelisable nilmanifolds with unobstructed deformations.
 \end{enumerate}
\end{cor}
Part~\ref{it:cx.nilp3} of this corollary was proven for $\nu=2$ in \cite{rollenske2011kuranishi}.
To our knowledge, parts \ref{it:cx.nilp5} to \ref{it:cx.dim20} provide the first examples of complex parallelisable nilmanifolds with unobstructed deformations whose associated Lie algebra is not freely nilpotent. It is not surprising that the examples where the Lie algebra is not freely nilpotent where not found in the more experimental calculations of Kuranishi spaces in \cite{rollenske2011kuranishi}: the one of lowest dimension is a $5$-step nilpotent Lie algebra of dimension~$10$, and the smallest example which is $4$-step nilpotent has dimension~$17$.

It would be interesting to study if this class of complex parallelisable manifolds has other distinguishing features in terms of the existence of special hermitian metrics, see for example \cite{fino-paradiso2023} and references therein.

We do not expect that similar characterisations of unobstructedness can be obtained in the more general category of nilmanifolds with left-invariant complex structure.

The paper is divided into two parts, which can be read independently:
in Section~\ref{sect: verbal}, we introduce pseudo-free Lie algebras and prove our classification results summarised in Theorem~\ref{thm: pseudofree classification}. This part is purely algebraic and does not refer to complex geometry.
In Section~\ref{sect: geometry}, we recap the deformation theory of complex parallelisable nilmanifolds and finally prove our criterion for unobstructedness stated in Theorem~\ref{thm: main deformation}.

\subsection*{Acknowledgements}

Mu'taz Abumathkur studied the problem of characterising unobstructed complex parallelisable nilmanifolds in his Bachelor and Master theses supervised by the second author. He partially proved Corollary~\ref{cor: classification cx parall}~\ref{it:cx.nilp3} for $\nu=3$ with explicit computations using the Hall basis and his continued interest in the problem stimulated the first discussions eventually leading to the present paper. A discussion with Istv\'an Heckenberger greatly helped to put preliminary results into the correct context.
We gratefully acknowledge support by the DFG through grants RO~3734/4-1,  RO~3734/5-1, and PA~4744/1-1.

\section{Pseudo-free Lie algebras}\label{sect: verbal}

We begin this section with a  short account of the theory of varieties of Lie algebras, following \cite[Chapter~4]{bahturin2021identical}, which leads us to the notion of pseudo-free Lie algebras. Then, after recalling some basic facts from  representation theory, we prove our classification result.

We work over the field $\IC$. To simplify our notation, ordinary multiplication is used to denote the Lie bracket and is right-associative,
meaning that \[ w_1\cdots w_s=[w_1,[w_2,\cdots[w_{s-1},w_s]\cdots]] \;. \]

\subsection{Verbal ideals and varieties of Lie algebras}\label{sect: varieties}
A class $\IV$ of Lie algebras is called a \emph{variety} if there exists a collection $R$ of non-associative polynomials such that a Lie algebra $\lg$ is in $\IV$ if and only if $\lg$ satisfies the relations $\{ r = 0 \mid r\in R\}$.
More precisely, if we consider $R$ as a subset of a free Lie algebra $\lf$, then a Lie algebra $\lg$ is in $\IV$ if and only if $R$ lies in the kernel of every homomorphism $\lf\to\lg$.
\begin{exam}\label{exam:variety-nilpotent}
Consider the non-associative polynomial $r = x_1 \cdots x_{\nu+1}\in\lf$, where $\lf$ is the free Lie algebra generated by $x_1, \ldots, x_{\nu+1}$.
Let $\lg$ be a Lie algebra. Then $r$ is in the kernel of every homomorphism $\lf \to \lg$ if and only if $\lg$ is $\nu$-step nilpotent.
The variety $\IV$ associated to $\{r\}$ is the class of $\nu$-step nilpotent Lie algebras.
\end{exam}

Let us fix a variety $\IV$ associated to $R\subset\lf$.
Then every Lie algebra $\lg$ has a largest quotient $\lg(\IV)$ contained in $\IV$.
Concretely, $\lg(\IV)$ is the quotient by the ideal generated by $\varphi(r)$ for all $r\in R$ and all homomorphisms $\phi\colon\lf\to\lg$.

If $\lf_X$ is a free Lie algebra with generating set $X$, we call $\lf_X(\mathbb V)$ the free Lie algebra of $\IV$ with generating set $X$.
Free Lie algebras in $\IV$ satisfy the following universal property, see \cite[Prop.~4.1]{bahturin2021identical}.
\begin{lem}\label{lem:free-univ}
Let $\lf_X(\IV)\in\IV$ be the free Lie algebra on $X$. Then
every map $X\to\lg$ to some $\lg\in\IV$ extends uniquely to a homomorphism $\lf_X(\IV)\to\lg$.
\end{lem}
\begin{proof}
A map $X\to\lg$ extends uniquely to a homomorphism $\lf_X\to\lg$.
Since $\lg\in\IV$, it follows that $\varphi(r)$ is in the kernel of this map for all $r\in R$ and all homomorphisms $\phi\colon\lf\to\lf_X$.
Therefore, we obtain a (uniquely determined) homomorphism $\lf_X(\IV)\to\lg$ extending $X\to\lg$.
\end{proof}
In particular, the free Lie algebra of $\IV$ with generating set $X$ is fully characterised by the universal property in Lemma~\ref{lem:free-univ}.

Clearly, we have $\lf_X(\IV)\cong\lf_Y(\IV)$ if $X$ and $Y$ have the same cardinality.
By \cite[Prop.~4.2]{bahturin2021identical}, the converse is true as well.
Therefore, we may define the free Lie algebra of $\IV$ of rank~$m$ to be $\lf_X(\IV)$ with $X=\{x_1,\ldots,x_m\}$.

Let $\lg$ be a finitely generated nilpotent Lie algebra. We call $\lg$ \emph{pseudo-free} if $\lg$ appears as a free Lie algebra of some variety $\IV$.

\begin{exam}\label{exam:nilp-pseudo-free}
A freely nilpotent Lie algebra $\gothn_{m,\nu}$ is pseudo-free, since $\gothn_{m,\nu}$ is the free Lie algebra of rank~$m$ inside the variety of $\nu$-step nilpotent Lie algebras considered in Example~\ref{exam:variety-nilpotent}.
\end{exam}

An ideal $\lh$ of a free Lie algebra $\lf$ is called \emph{verbal} or \emph{fully invariant}\footnote{In the theory of associative algebras the analogous concept is known as a T-ideal.}
if $\varphi(\lh)\subset\lh$ for all endomorphisms $\varphi\colon\lf\to\lf$.

The following lemma characterises pseudo-free Lie algebras $\lg$ in terms of an intrinsic condition on $\lg$:

\begin{lem}\label{lem: pseudo-free}
	Let $\lg$ be a nilpotent Lie algebra with minimal set of generators $X=\{x_1,\ldots,x_m\}$.
    Let $V$ be the linear subspace spanned by $X$, so that we have $V\cong\lg/[\lg,\lg]$ under the projection $\lg\to\lg/[\lg,\lg]$.
	Let $\lf$ be the free Lie algebra on $X$.
	Let $\lh\subset\lf$ be the kernel of $\lf\to\lg$, inducing a presentation $\lg=\lf/\lh$.
    Then the following statements are equivalent:
	\begin{enumerate}
		\item\label{it:pseudo-free} The Lie algebra $\lg$ is pseudo-free, that is, $\lg$ is the free Lie algebra of rank~$m$ of some variety $\IV$.
		\item\label{it:univ-ext-self} Every map $X\to\lg$ extends (uniquely) to an endomorphism $\lg\to\lg$. Equivalently, every linear map $V\to\lg$ extends (uniquely) to an endomorphism $\lg\to\lg$.
		\item\label{it:verbal} The ideal $\lh$ is verbal, that is, we have $\varphi(\lh)\subset\lh$ for all endomorphisms $\varphi\colon\lf\to\lf$.
	\end{enumerate}
\end{lem}
\begin{proof}
We prove three implications.
\begin{prooflist}
\item[$\text{\ref{it:pseudo-free}}\implies\text{\ref{it:univ-ext-self}}$]
Applying the universal property from Lemma~\ref{lem:free-univ} to $\lg\in\IV$ itself, we immediately see that \ref{it:pseudo-free} implies \ref{it:univ-ext-self}.
\item[$\text{\ref{it:univ-ext-self}}\implies\text{\ref{it:verbal}}$]
An endomorphism $\varphi\colon\lf\to\lf$ corresponds to a map $X\to\lf$. After composing with the projection $\lf\to\lg$, we can extend the map $X\to\lg$ by \ref{it:univ-ext-self} to an endomorphism $\lg\to\lg$, yielding a commutative diagram
\[
\begin{tikzcd}
 \lf \rar{\phi}\dar & \lf\dar\\
 \lg \rar & \lg 
\end{tikzcd}.
\]
In other words, $\varphi$ descends to an endomorphism of $\lg=\lf/\lh$, meaning that $\varphi(\lh)\subset\lh$.
\item[$\text{\ref{it:verbal}}\implies\text{\ref{it:pseudo-free}}$]
We consider the variety $\IV$ defined by the relations $\lh\subset\lf$.
Since $\varphi(\lh)\subset\lh$ for all endomorphisms $\phi$, it follows that $\lf(\IV)$ is by definition just $\lf/\lh=\lg$.
\qedhere
\end{prooflist}
\end{proof}

The following lemma is useful to decide whether two pseudo-free Lie algebras are isomorphic.
\begin{lem}\label{lem:pseudo-free isomorphic}
Let $\lg$ and $\lg'$ be pseudo-free Lie algebras with generators $x_1,\ldots,x_m$ and $y_1,\ldots,y_m$, respectively.
If $\lg\cong\lg'$, then there already exists an isomorphism $g\colon\lg\to\lg'$ such that $g(x_i)=y_i$.
In particular, two quotients $\lf/\lh$ and $\lf/\lh'$ of a free Lie algebra $\lf$ by verbal ideals $\lh,\lh'\subset\lf$
are isomorphic if and only if $\lh=\lh'$.
\end{lem}
\begin{proof}
Let $f\colon\lg\to\lg'$ be an isomorphism.
By property~\ref{it:univ-ext-self} of Lemma~\ref{lem: pseudo-free},
there exists an endomorphism $h\colon\lg\to\lg$ such that $h(x_i)=f^{-1}(y_i)$.
Hence, $g=f\circ h\colon\lg\to\lg'$ is a homomorphism satisfying $g(x_i)=y_i$.
Analogously, there exists a homomorphism $g'\colon\lg'\to\lg$ such that $g'(y_i)=x_i$.
The endomorphisms $g'\circ g$ and $g\circ g'$ fix the generators $x_1,\ldots,x_m$ and $y_1,\ldots,y_m$, respectively,
and are thus the identities on $\lg$ and $\lg'$, meaning that $g$ and $g'$ are inverse to each other.

The claim on quotients by verbal ideals follows using property~\ref{it:verbal} of Lemma~\ref{lem: pseudo-free} and considering the generating sets induced by $\lf$.
\end{proof}

\subsection{Representation theory}

The aim of this section is to connect the study of verbal ideals to the study of irreducible $\Gl(V)$-representations. Following \cite[Chapter~3]{bahturin2021identical} and \cite[Chapter~8]{Fulton_1996} we first review some well-known facts about representation theory in our setting.

Let $\lf$ be a free Lie algebra over $X= \{x_1,\dots,x_m\}$. As a consequence of the Jacobi identity, every element in $\lf$
can be written as a linear combination of monomials $x_{i_1}\cdots x_{i_n}$ with $i_1,\ldots,i_n\in\{1,\ldots,m\}$.
We say that such a monomial $x_{i_1}\cdots x_{i_n}$ has multidegree $(d_1,\ldots,d_m)$ if $j$ appears exactly $d_j$~times
among the indices $i_1,\ldots,i_n$ for each $j\in\{1,\ldots,m\}$. In particular, $d_1+\cdots+d_m=n$.
Since the Jacobi identity respects this multigrading, $\lf$ decomposes as a vector space into the direct sum of multihomogeneous pieces
$\lf_{(d_1,\ldots,d_m)}$:
\begin{equation}\label{eq: multideg} \lf=\bigoplus_{n\ge1} \lf_n, \quad\text{where}\quad \lf_n=\bigoplus_{d_1+\cdots+d_m=n}\lf_{(d_1,\ldots,d_m)}. \end{equation}
Let $V= \langle x_1, \dots,x_m\rangle$ be the vector space generated by $X$. Then every automorphism in $\Gl(V)$ can be extended to a Lie algebra automorphism of $\lf$ and we obtain a representation $\Gl(V) \to \Aut(\lf)$ preserving the homogeneous components of the decomposition $\lf = \bigoplus_{n\geq 1} \lf_n$.
Indeed, the multidegree decomposition \eqref{eq: multideg} is exactly the weight space decomposition of the finite-dimensional polynomial representation of $\Gl(V)$ on $\lf_n$. We are interested in a decomposition of this representation into a finite direct sum of irreducible representations.

It is well known that such irreducible representations can be  described in terms of partitions of $n$: 
let $\lambda \vdash n$ be a partition of $n$, that is, a sequence $\lambda = (\lambda_1, \dots, \lambda_m)$ of weakly decreasing non-negative integers such that $\lambda_1 + \dots + \lambda_m = n$. Then, as described in \cite[Chapter~8]{Fulton_1996}, we can associate to $\lambda$ an irreducible representation  of the symmetric group $S_n$, the so-called Specht module $T_\lambda$. The corresponding $\Gl(V)$-representation is the  Schur module
\[
V_\lambda = V^{\otimes n} \otimes_{\IC S_n} T_\lambda,
\]
where $\IC S_n$ denotes the complex group algebra. We often suppress the zero entries of $\lambda$ when referring to the representation $V_\lambda$. The proof of the following theorem can be found in \cite[p.\ 114]{Fulton_1996}.

\begin{thm} Let $V$ be a complex vector space of dimension $m$.
	\begin{enumerate}
		\item The representation $V_\lambda$ of $\Gl(V)$ is an irreducible representation of highest weight $(\lambda_1, \dots, \lambda_m)$.
		 In particular, two representations $V_\lambda$ and $V_\mu$ are isomorphic if and only if $\lambda = \mu$.
		\item The $V_\lambda$ are all irreducible polynomial representations of $\Gl(V)$.
		
	\end{enumerate}
\end{thm}
As a consequence, every  homogeneous component $\lf_n$ admits a decomposition
\[
\lf_n = \bigoplus_{\lambda\vdash n} m_\lambda V_\lambda
\]
for certain multiplicities $m_\lambda$. 
A complete list of the multiplicities $m_\lambda$ up to $n =10$ can be found in \cite{thrall1942symmetrized}. In the following proposition, we only provide the multiplicities relevant to our considerations.

\begin{prop}\label{prop:mult}
 Let $\lf$ be a free Lie algebra on $m$ generators. Then its components $\lf_n$ of degree $n\leq 7$ decompose as $\Gl(V)$-representations in the following way:
	\begin{align*}
		\lf_2 &= V_{(1,1)},\\
		\lf_3 &= V_{(2,1)},\\
		\lf_4 &= V_{(3,1)} \oplus V_{(2,1,1)},\\
		\lf_5 &= V_{(4,1)} \oplus V_{(3,2)} \oplus V_{(3,1,1)} \oplus V_{(2,2,1)} \oplus V_{(2,1,1,1)},\\
		\lf_6 &= V_{(5,1)} \oplus V_{(4,2)} \oplus V_{(3,3)} \oplus 2V_{(4,1,1)} \oplus 3V_{(3,2,1)} 
		\\
		& \qquad\qquad\qquad 
		\oplus V_{(3,1,1,1)} \oplus 2V_{(2,2,1,1)} \oplus V_{(2,1,1,1,1)},\\
		\lf_7 &= V_{(6,1)} \oplus 2V_{(5,2)} \oplus 2V_{(4,3)} \oplus 2V_{(5,1,1)}\oplus 5V_{(4,2,1)} \oplus \cdots,
	\end{align*} 
where $V_{(\lambda_1, \dots, \lambda_r)}$ is the irreducible representation defined above for $r\leq m$ and for $r>m$ we have $V_{(\lambda_1, \dots, \lambda_r)} = 0$.
\end{prop}
\begin{proof}
	The multiplicities $m_\lambda$ can be computed with the following identity proven in \cite[Prop.~3.22]{bahturin2021identical}: let $\lambda$ be a partition of $n$ and let $\chi$ be the character of the corresponding $S_n$-representation $T_\lambda$. Then
	\[
	m_\lambda = \frac{1}{n}\sum_{d\mid n}\mu(d)\chi(\tau^{n/d}),
	\]
	where $\tau\in S_n$ is a cycle of length $n$ and $\mu$ is the Möbius function.
	The characters in the above formula can be evaluated with the so-called Murnaghan--Nakayama rule explained in \cite[Section~7.17]{Stanley_Fomin_1999}.
\end{proof}

\begin{rem}\label{rem: mult} Schocker determined all partitions $\lambda\vdash n$ with $m_\lambda =1$ in  \cite{schocker2003embeddings}.
Besides the cases already appearing in the above proposition the only partitions with $m_\lambda=1$ are $(n-1,1)$ and $(2,1,\dots,1)$ and the two partitions $(2,2,2,2)$ and $(4,4)$ for $n=8$.
\end{rem}

 Let us return to the study of verbal ideals. We denote by $\lf_{\geq\nu}$ the ideal  $\bigoplus_{n\geq\nu} \lf_n$.
\begin{prop}\label{prop:verbal}
	Let $\lf$ be  a free algebra generated by $m$ elements. Then the following hold:
	\begin{enumerate}
		\item\label{it:verbal.1} If $\lh \subset \lf$ is a verbal ideal containing $\lf_{\geq \nu+1}$, then the ideal 
		\[\lh/\lf_{\geq \nu+1}\subset \lf/\lf_{\geq \nu+1}=\frn_{m,\nu}\]
		is $\Gl(V)$-invariant.
		\item\label{it:verbal.2} If the multiplicities $m_\lambda$ in the decomposition $\lf/ \lf_{\geq \nu+1} = \bigoplus m_\lambda V_\lambda$ are at most~$1$, then there exist only finitely many verbal ideals $\lh \subset \lf$ containing $\lf_{\geq \nu+1}$.
		\item\label{it:verbal.3} The verbal ideals $\lh\subset \lf$ satisfying $\lf_{\geq\nu+1} \subset \lh \subset \lf_{\geq\nu}$ are precisely the ideals of the form $U \oplus \lf_{\geq\nu+1}$, where $U$ is a $\Gl(V)$-invariant subspace of $\lf_\nu$.
	\end{enumerate}
\end{prop}
\begin{proof}
	By definition, an ideal is verbal if and only if it is preserved under every endomorphism of $\lf$. In particular, every verbal ideal is $\Gl(V)$-invariant. This proves \ref{it:verbal.1}.
	
	By Schur's Lemma, the decomposition of a $\Gl(V)$-representation into irreducible representations is unique if all multiplicities are at most~$1$. Combined with \ref{it:verbal.1} this implies \ref{it:verbal.2}.

	Clearly, every ideal $\lh$ with $\lf_{\geq \nu+1} \subset \lh \subset \lf_{\geq \nu}$ which is preserved by all endomorphisms of $\lf$ is of the form $U \oplus \lf_{\geq \nu+1}$ for some $\Gl(V)$-invariant subspace $U \subset \lf_\nu$. It remains to show that every ideal of the form $U \oplus \lf_{\geq \nu+1}$ is verbal. Since the action of $\End(V)$ on $\lf$ is continuous and $\Gl(V)\subset\End(V)$ is dense, every $\Gl(V)$-invariant subspace of $U \subset\lf_\nu$ is also $\End(V)$-invariant. Together with the fact that the ideal $\lf_{\geq \nu}$ is verbal, this implies that every ideal of the form $U \oplus \lf_{\geq \nu+1}$ is preserved by all endomorphisms of $\lf$ and is thus verbal.
	\end{proof}

\begin{cor}\label{cor:3step}
A nilpotent Lie algebra of nilpotency index $\nu\leq 3$ is pseudo-free if and only if it is freely nilpotent.
\end{cor}
\begin{proof}
	We have already seen in Example~\ref{exam:nilp-pseudo-free} that every freely nilpotent Lie algebra is pseudo-free.
    By Lemma~\ref{lem: pseudo-free}, a pseudo-free Lie algebra on $m$ generators is the quotient of $\lf$ by a verbal ideal.
	By Proposition~\ref{prop:mult}, the induced $\Gl(V)$-representations on $\lf_2$ and $\lf_3$ are irreducible. Thus Proposition~\ref{prop:verbal} implies that the only pseudo-free Lie algebras with $\nu \leq 3$ are freely nilpotent. 
\end{proof}

\begin{cor}\label{cor: tuple}
For every tuple $(m,\nu)$ with $\nu \leq 5$ or $(m,\nu) = (2,6)$ there exist only finitely many pseudo-free Lie algebras with $m$ generators of nilpotency index $\nu$.
\end{cor}
\begin{proof}
	This follows directly from Proposition~\ref{prop:verbal}~\ref{it:verbal.2} and  Proposition~\ref{prop:mult}.
\end{proof}

\begin{exam}\label{exam: smallest}
Consider the free $7$-step nilpotent Lie algebra $\frn_{2,7}$ on two generators and decompose (non-uniquely) $\lf_7$ into irreducible representations as in Proposition~\ref{prop:mult}. Now pick highest weight vectors $v_1$ and $v_2$ for the two representations isomorphic to $V_{(5,2)}$, which can be chosen to be non-associative polynomials with rational coefficients in the generators $x_1$ and $x_2$. Then for any $\mu = (\mu_1:\mu_2)\in \IP^1_\IC$ we get an irreducible representation $U_\mu\subset \lf_7$, also isomorphic to $V_{(5,2)}$, generated by the linear combination $\mu_1 v_1 + \mu_2 v_2$. Consider the ideal
\[
\mathfrak a_\mu = V_{(5,1)} \oplus V_{(6,1)} \oplus 2V_{(5,2)} \oplus U_\mu \subset \frn_{2,7}.
\]
Then  $\{\lg_\mu = \frn_{2,7}/\mathfrak a_\mu\}_{\mu\in \IP^1_\IC}$ defines a family of $20$-dimensional pseudo-free Lie algebras.
The Lie algebras in this $1$-parameter family are pairwise non-isomorphic by Lemma~\ref{lem:pseudo-free isomorphic}.
\end{exam}

\begin{exam}\label{exam: 6step, 3gen}
Consider the free $6$-step nilpotent Lie algebra $\frn_{m,6}$ on $m\geq 3$ generators. Since the representation $V_{(4,1,1)}$ appears with multiplicity 2 we can define a 1-parameter family of irreducible representations $V_{(4,1,1)}\isom U_\mu \subset \lf_6$ as in Example~\ref{exam: smallest}. The Lie algebras $\frn_{m,6}/U_\mu$ provide infinitely many non-isomorphic pseudo-free Lie algebras with $m$ generators and $\nu =6$.
\end{exam}

\begin{exam}\label{exam: 7 step always infinite}
Combining Remark \ref{rem: mult} with \cite[Proposition~2]{klyachko1974lie} shows that for $n \geq 7$ the irreducible representation corresponding to $(n-2,2)$ always appears with multiplicity at least $2$. Hence, by the same construction as in Example~\ref{exam: smallest} one can construct infinitely many pseudo-free Lie algebras for every tuple $(m,\nu)$ not listed in Corollary~\ref{cor: tuple}.
\end{exam}

\subsection{The Lie algebra of derivations}
In Proposition~\ref{prop:verbal}~\ref{it:verbal.3} we have seen that pseudo-free nilpotent Lie algebras with relations appearing only in the homogeneous component of highest degree can be described purely in terms of representation theory. To control more general pseudo-free Lie algebras
 we recall some results by Zhuravlev \cite{zhuravlev1997}.

Let $R$ be a set of non-associative polynomials, which determines a variety of Lie algebras $\mathbb V$. Let $\lf$ be a free Lie algebra with finite generating set $X$. Then $R$ determines a unique verbal ideal $\lh \subset \lf$ with the property that $\lf(\mathbb V) = \lf/\lh$.
The ideal $\lh$ is homogeneous as well, because it is invariant under the action of $\Gl(V)$, so we can write $\lh = \bigoplus_n \lh_n$ where $\lh_n = \lh\cap \lf_n$.

We denote by $\mathcal{U} =\mathcal{U}(\Der (\lf))$ the universal enveloping algebra of the Lie algebra  $\Der (\lf)$ of derivations of $\lf$. Considering $R$ as a subset of $\lf$ we can define the ideal
\[
 \mathcal{U}(R) = \langle\delta(r) \,|\, r \in R,\, \delta \in \mathcal{U}\rangle\subset \lf.
\]
The main theorem of \cite{zhuravlev1997} states that in the above situation we have
	\(
	\lh = \mathcal{U}(R).
	\)
	
This alternative description of verbal ideals makes it possible to compute the verbal ideal generated by a sum of  irreducible representations $V_\lambda$ of some fixed degree.
To this purpose we consider the grading
\[
\Der( \lf) = \bigoplus_{n\geq 0} I_n,
\]
where $I_n =\{ \delta \in \Der (\lf) \, | \, \delta(X) \subset \lf_{n+1}\}$. The ideal \[ I = \bigoplus_{n\ge1} I_n = \{\delta \in \Der \lf \, | \, \delta(X) \subset [\lf, \lf]\} \] induces a grading on its universal enveloping algebra
\[
\mathcal{U}(I) = \bigoplus_{n \geq 1} U_n = I_1 \oplus (I_2 + I_1I_1) \oplus \cdots \,.
\]
As shown in \cite[Sect.~4]{zhuravlev1997} we get the following.
\begin{prop}\label{prop:derivations}
If $R$ is a set of non-associative polynomials of degree $n$ with associated verbal ideal $\lh \subset \lf$, then
\[
 \lh_{n+k} = U_k(\lh_n) 
\]
for all $k\geq 0$. In particular, we have $\lh_{n+1} = I_1(\lh_n)$ and $\lh_{n+2} = I_2(\lh_n) + I_1I_1(\lh_n)$.
\end{prop}

\begin{exam}\label{exam:derivations}
	Let $\lf$ be a free Lie algebra on two generators $x,y$ and let $V=\langle x,y\rangle$. We  compute the verbal ideal $\lh$ generated by $V_{(4,1)}\subset \lf_5$.
	
	The element $xxxxy$ is a highest weight vector for $V_{(4,1)}$. As we later see in Table~\ref{tab: dimensions}, we have $\dim V_{(4,1)}=4$. Therefore, a basis of $V_{(4,1)}$ is given by
	\[
	v_1 = xxxxy,\, v_2 =xxyxy+xyxxy+yxxxy,\, v_3=yyxxy+yxyxy+xyyxy, \,v_4 = yyyxy.
	\]
	By the previous corollary we have $\lh_6 = U_1(V_{(4,1)}) = I_1(V_{(4,1)})$. Since $\lf$ is generated by $x$ and $y$, we have $I_1= \langle \ad x, \ad y\rangle$ and therefore $\lh_6$ is generated by the vectors $ [x,v_1], [y,v_1],\dots,[x,v_4], [y,v_4]\in \lf_6$. One can easily check that these vectors are linearly independent and thus $\dim \lh_6 =8$. Comparing with the dimensions in Table~\ref{tab: dimensions} we see that $\lh_6 = V_{(5,1)}\oplus V_{(4,2)}$ is the only possibility.
	Similarly, one shows $\lh_7 =U_2(V_{(4,1)}) = \lf_7$. Therefore, the verbal ideal generated by $V_{(4,1)}$ is the ideal $V_{(4,1)} \oplus V_{(5,1)} \oplus V_{(4,2)} \oplus \lf_{\geq 7}\subset \lf$.
\end{exam}

\begin{exam}\label{exam: 7-step, 2gen}
	Consider the free 7-step nilpotent Lie algebra $\frn_{2,7}$ on two generators. Then the verbal ideal $\mathfrak{a}\subset \frn_{2,7}$ generated by $\lh_6 =V_{(5,1)}\oplus V_{(4,2)} \subset \lf_6$ is 
	\[
		\mathfrak{a} \isom V_{(5,1)}\oplus V_{(4,2)} \oplus V_{(6,1)} \oplus 2V_{(5,2)} \oplus V_{(4,3)}.
	\]
	Note that while $\lf_7$ contains two copies (and thus a 1-parameter family) of the representation  $V_{(4,3)}$, in the above decomposition we have to take the unique one that is contained in $I_1(\lh_6)$. 
In particular, we get just one Lie algebra and not a family as in Example~\ref{exam: smallest}.
\end{exam}

\begin{table}
	\caption{Non-abelian pseudo-free Lie algebras of nilpotency index $\nu \leq 5$, minimally generated by $m$ elements}
	\label{tab: nilpotency index 5}
\begin{tabular}{rrl}
			\toprule
			$\nu$& $m$ &Lie algebras 
			\quad  (in each row $\epsilon_i \in \{0,1\}$ but not all equal to $1$) 
			\\
			\midrule
			$2$& $\geq 2$& $\frn_{m,2}$ \\
			$3$& $\geq 2$ & $\frn_{m,3}$ \\
			\midrule
			\multirow{2}*{$4$} & $2$ & $\frn_{2,4}$\\
			 & $\geq 3$ &  $\frn_{m,4}$, $\frn_{m,4}/V_{(3,1)}$, $\frn_{m,4}/ V_{(2,1,1)}$\\
			\midrule
			& $2$ & $\frn_{2,5}$, $\frn_{2,5}/V_{(4,1)}$, $\frn_{2,5}/ V_{(3,2)}$\\
			$5$ & $3$ &$\frn_{3,5}/\left( \epsilon_1 V_{(4,1)} \oplus \epsilon_2 V_{(3,2)} \oplus \epsilon_3 V_{(3,1,1)} \oplus\epsilon_4 V_{(2,2,1)} \right) $ \\
			& $3$&$\frn_{3,5}/\left(V_{(2,1,1)} \oplus  V_{(3,2)} \oplus V_{(3,1,1)} \oplus V_{(2,2,1)}
			\right) $\\
			\midrule
			\multirow{2}*{$5$}& \multirow{2}*{$\geq 4$} &$\frn_{m,5}/\left( \epsilon_1 V_{(4,1)} \oplus \epsilon_2 V_{(3,2)} \oplus \epsilon_3 V_{(3,1,1)} \oplus\epsilon_4 V_{(2,2,1)} \oplus \epsilon_5 V_{(2,1,1,1)}\right) $ \\
			 & &$\frn_{m,5}/\left(V_{(2,1,1)} \oplus  V_{(3,2)} \oplus V_{(3,1,1)} \oplus V_{(2,2,1)} \oplus V_{(2,1,1,1)}
			\right) $\\
			\bottomrule
	\end{tabular}
\end{table}

\begin{prop}\label{prop: 5-step classification}
Let $\lg$ be a non-abelian pseudo-free Lie algebra of nilpotency index $\nu \leq 5$, minimally generated by $m$~elements. Then $\lg$ is isomorphic to exactly one of the Lie algebras in Table~\ref{tab: nilpotency index 5}.
\end{prop}
\begin{proof}
Let us write $\lg=\lf/\lh$ for some verbal ideal $\lh$ in the free Lie algebra $\lf$ on $m$~generators.
Since $\lg$ is $\nu$-step nilpotent, we have $\lf_{\ge\nu+1}\subset\lh$.
 By Corollary~\ref{cor: tuple}, there exist only finitely many choices for $\lh$. If $\lh\subset\lf_{\ge\nu}$, then Proposition~\ref{prop:verbal} implies that $\lh_\nu$ is a direct sum of (possibly none, but not all) irreducible $\Gl(V)$-subrepresentations of $\lf_\nu$.
 
It remains to treat the case where $\nu=5$ and $\lh_4\neq 0$. This can only happen for $m\ge3$,
since $\lf_4$ is irreducible for $m=2$ and we would thus have $\lh_5=\lf_5$ otherwise.
For $m\ge3$, the two choices for $\lh_4$ are $\lh_4 = V_{(3,1)}$ and $\lh_4 = V_{(2,1,1)}$.

We start with the case $\lh_4 = V_{(3,1)}$ and claim that $\lh_5 = I_1(V_{(3,1)}) = \lf_5$. Since the highest weight vectors of every irreducible subrepresentation of $\lf_5$ involve at most $4$~generators, it is enough to show this for $m\leq 4$. This can  be done as in Example~\ref{exam:derivations}.

By the same argument as above it suffices to treat the case $\lh_4 = V_{(2,1,1)}$ with $ m \leq 4$. Computing $I_1(V_{(2,1,1)})$ yields $\lh_5 = V_{(3,2)} \oplus V_{(3,1,1)} \oplus V_{(2,2,1)}$ for $m=3$ and $\lh_5 = V_{(3,2)} \oplus V_{(3,1,1)} \oplus V_{(2,2,1)} \oplus V_{(2,1,1,1)}$ for $m \geq 4$. 
\end{proof}

\subsection{Classification in low dimensions}\label{sect: class}
In this section we classify pseudo-free Lie algebras up to dimension $20$, which is the smallest dimension where infinitely many non-isomorphic pseudo-free Lie algebras occur. 
To control the dimensions of the pseudo-free Lie algebras we find, we make use of two important dimension formulas:
\begin{enumerate}
 \item The dimensions of the irreducible representations $V_\lambda$ can be computed with the Weyl dimension formula \cite[Section~3.3.5]{bahturin2021identical}: for an $m$-dimensional vector space $V$ and a partition $\lambda = (\lambda_1, \dots, \lambda_m)$ we have
\[
\dim V_\lambda = \prod_{1\leq i<j \leq m} \frac{\lambda_i-\lambda_j +j-i}{j-i}.
\]
\item For a free Lie algebra $\lf$ on $m$ generators the dimension of its degree~$n$ component $\lf_n$ can be computed with Witt's dimension formula \cite[Theorem~3.2]{bahturin2021identical} and we get
\[
\dim \frn_{m,\nu} = \sum_{n\leq \nu} \dim \lf_n = \sum_{n\leq \nu}\frac{1}{n}\sum_{d\mid n}\mu(d)m^{n/d}.
\]
\end{enumerate}
The  dimensions relevant for our discussion are listed in Table~\ref{tab: dimensions}.
\begin{table}
	\caption{Dimension data for $\Gl(m,\IC)$-representations $V_\lambda$ and free nilpotent Lie algebras}
	\label{tab: dimensions}
	\subfloat{	
		\begin{tabular}{l rrrr}
			\toprule
			& \multicolumn{4}{c}{$m$} \\
			\cmidrule{2-5}
			Partition $\lambda$ & \phantom{00}2 & \phantom{00}3 & \phantom{00}4 & \phantom{00}5 \\
			\midrule
			$(3,1)$ & 3 & 15 & 45 & 105 \\
			$(2,1,1)$ & 0 & 3 & 15 & 45 \\
			\midrule
			$(4,1)$ & 4 & 24 & 84 & 224	\\
			$(3,2)$ & 2 & 15 & 60 & 175 \\
			$(3,1,1)$ & 0 & 6 & 36 & 126 \\
			$(2,2,1)$ & 0 & 3 & 20 & 75	\\
			$(2,1,1,1)$ & 0 & 0 & 4 & 24 \\ 
			\midrule
			$(5,1)$ & 5 & 35 & 140 & 420 \\
			$(4,2)$ & 3 & 27 & 126 & 420 \\
			$(3,3)$ & 1 & 10 & 50 & 175 \\ 
			$(4,1,1)$ & 0 & 10 & 70 & 280 \\
			$(3,2,1)$ & 0 & 8 & 64 & 280\\ 
			\midrule
			$(6,1)$ & 6 & 48 & 216 & 720 \\
			$(5,2)$ & 4 & 42 & 224 & 840 \\
			$(4,3)$ & 2 & 24 & 140 & 560 \\
			\bottomrule
		\end{tabular}
	}
	\hfill
	\subfloat{
		\begin{tabular}{l rrrr}
			\toprule
			& \multicolumn{4}{c}{$m$}\\
			\cmidrule{2-5}
			Lie algebra & \phantom{000}2 & \phantom{000}3 & \phantom{000}4 & \phantom{000}5 \\
			\midrule
			$\frn_{m,2}$ & $3$ & $6$ & $10$ & $15$ \\
			$\frn_{m,3}$ & $5$ & $14$ & $30$ & $55$ \\
			$\frn_{m,4}$ & $8$ & $32$ & $90$ & $205$ \\
			$\frn_{m,5}$ & $14$ & $80$ & $294$ & $829$ \\
			$\frn_{m,6}$ & $23$ & $196$ & $964$ & \\
			$\frn_{m,7}$ & $41$ & $508$ & $3304$ & \\
			$\frn_{m,8}$ & $71$ & $1318$ & & \\
			\bottomrule
	\end{tabular}}
\end{table}

\begin{thm}\label{thm: class}
Every non-abelian pseudo-free Lie algebra of dimension at most~$20$ is isomorphic to one (and only one) of the Lie algebras listed in Table~\ref{tab: small dim}.
In particular, $20$ is the smallest dimension where infinitely many non-isomorphic pseudo-free Lie algebras exist.
\end{thm}
\begin{proof}
	By Proposition~\ref{prop:verbal} every pseudo-free nilpotent Lie algebra is obtained as the quotient of a freely nilpotent Lie algebra $\frn_{m,\nu}$ by a direct sum of the irreducible representations $V_\lambda$. Computing the dimension of the Lie algebras $\frn_{m,\nu}$ using Witt's formula yields $\dim \frn_{m,2} > 20$ for $m\geq 6$. By Corollary~\ref{cor:3step} every non-abelian pseudo-free Lie algebra with at least $6$ generators has dimension at least $21$. Hence, we only have to consider quotients of the Lie algebras $\frn_{m,\nu}$ for $m\leq 5$. We treat three cases.
\begin{prooflist}
\item[Case~1] For $4\leq m \leq 5$ we get $\dim \frn_{4,2} = 10$, $\dim \frn_{5,2}= 15$ and $\dim \frn_{m,\nu} > 20$ for $\nu \geq 3$. Therefore, by Corollary~\ref{cor:3step} the only pseudo-free Lie algebras with at least $4$ generators and of dimension at most $20$ are $\frn_{4,2}$ and $\frn_{5,2}$.
	
\item[Case~2] For $m=3$ we have $\dim \frn_{3,2} = 6$, $\dim \frn_{3,3} = 14$ and $\dim \frn_{3,4} = 32$. By Proposition~\ref{prop:mult} the degree~$4$ component of $\frn_{3,4}$ decomposes as $V_{(3,1)}\oplus V_{(2,1,1)}$. From Table~\ref{tab: dimensions} we get $\dim \frn_{3,4}/ V_{(3,1)} = 17$ and $\dim \frn_{3,4}/V_{(2,1,1)}=29$. 
Using the classification of pseudo-free Lie algebras of nilpotency index $\nu \leq5$ from Proposition~\ref{prop: 5-step classification} we see that any other pseudo-free Lie algebra with $m=3$ has larger dimension.
\item[Case~3] It remains to classify pseudo-free Lie algebras with $m=2$. By Corollary~\ref{cor: tuple} there are finitely many of these Lie algebras with nilpotency index at most~$6$. In order to determine the quotients of $\frn_{2,6}$ which are pseudo-free we compute
$U_1(V_{(3,1)}) = \lf_5$, $U_1(V_{(4,1)}) = V_{(5,1)} \oplus V_{(4,2)}$, and $U_1(V_{(3,2)}) = V_{(4,2)} \oplus V_{(3,3)}$.	 

	Therefore, by Proposition~\ref{prop:derivations}, Table~\ref{tab: 2 gen} contains all the possible pseudo-free Lie algebras with two generators of nilpotency index at most~$6$. In particular, the only Lie algebras of dimension larger than $20$ are $\frn_{2,6}$ and $\frn_{2,6}/V_{(3,3)}$. Next, we cover Lie algebras with $\nu =7$. Again we compute
	\begin{align*}
		U_2(V_{(4,1)}) &= \lf_7,\\
		U_2(V_{(3,2)}) &= 2V_{(5,2)} \oplus 2V_{(4,3)},\\
		U_1(V_{(5,1)}\oplus V_{(4,2)}) &= V_{(6,1)} \oplus 2 V_{(5,2)} \oplus V_{(4,3)} ,\\
		U_1(V_{(5,1)}\oplus V_{(3,3)}) &= V_{(6,1)} \oplus V_{(5,2)} \oplus V_{(4,3)},\\
		U_1(V_{(4,2)}\oplus V_{(3,3)}) &= V_{(5,2)} \oplus 2V_{(4,3)}.
	\end{align*}
	Since $\dim \frn_{2,7} = 41$, the only 7-step pseudo-free Lie algebras relevant to our classification are the 17-dimensional Lie algebra $\frn_{2,7}/\gotha$ from Example~\ref{exam: 7-step, 2gen} and the Lie algebras $\lg_\mu$ from Example~\ref{exam: smallest}. Finally, suppose $\lg$ is a pseudo-free Lie algebra with $\nu \geq 8$ and of dimension at most~$20$. Then
	$
	\lg/ \lg_{\geq8} \cong \frn_{2,7}/\gotha
	$.
	However, we also have $U_2(V_{(5,1)} \oplus V_{(4,2)}) = \lf_8$ and thus $\lg$ has nilpotency index 7. \qedhere
\end{prooflist}
\end{proof}

\begin{table}
	\caption{Non-abelian pseudo-free Lie algebras of nilpotency index $\nu \leq 6$, minimally generated by $m=2$ elements}
	\label{tab: 2 gen}
	\begin{tabular}{rlr}
		\toprule
		$\nu$& Lie algebra & \llap{Dimension} \\
		\midrule
		$2$ & $\frn_{2,2}$ & $3$\\
		$3$ & $\frn_{2,3}$ & $5$\\
		$4$ & $\frn_{2,4}$ & $8$\\
		\midrule
		$5$ &    $\frn_{2,5}$   & $14$\\
		    &    $\frn_{2,5}/V_{(4,1)}$ & $10$\\
		    &    $\frn_{2,5}/V_{(3,2)}$ & $12$\\
		\midrule
		$6$ &   $\frn_{2,6} $ & $23$\\
		    &   $\frn_{2,6}/V_{(5,1)} $ & $18$\\
		    &   $\frn_{2,6}/ V_{(4,2)} $ & $20$\\
		    &   $\frn_{2,6}/V_{(3,3)}$ & $22$\\
		    &   $\frn_{2,6}/\left(V_{(4,2)} \oplus V_{(3,3)}\right)$ & $19$\\
		    &   $\frn_{2,6}/\left(V_{(5,1)}  \oplus V_{(3,3)}\right)$ & $ 17$\\
		    &   $\frn_{2,6}/\left(V_{(5,1)} \oplus V_{(4,2)} \right)$ & $15$\\
		    &   $\frn_{2,6}/\left(V_{(4,1)}\oplus V_{(5,1)} \oplus V_{(4,2)} \right)$ & $11$\\
		    &   $\frn_{2,6}/\left(V_{(3,2)}\oplus V_{(4,2)} \oplus V_{(3,3)}\right)$ & $17$\\
		\bottomrule
	\end{tabular}
\end{table}

\subsection{Proof of Theorem~\ref{thm: pseudofree classification}}
Collecting the results from this section, we can now prove our classification theorem for pseudo-free Lie algebras: the first item is Corollary~\ref{cor:3step}, the second item is Proposition~\ref{prop: 5-step classification}, and the third item is Theorem~\ref{thm: class}. \qed

\section{Deformations of complex parallelisable nilmanifolds}\label{sect: geometry}

In this section, we show how the previous results can be applied to study deformations of complex parallelisable nilmanifolds.

\subsection{Complex parallelisable nilmanifolds}

Let $G$ be a simply connected complex nilpotent Lie group with Lie algebra $\lg$.
By a criterion of Mal'cev \cite{malcev51} the Lie group $G$ admits a lattice $\Gamma \subset G$, i.\,e., a discrete cocompact subgroup, if and only if the real Lie algebra underlying $\lg$ can be defined over $\IQ$.
Since the multiplication in $G$ is holomorphic, the quotient $X:=\Gamma\backslash G$ by the left action of $\Gamma$ is a compact complex parallelisable nilmanifold.

It is sometimes useful to consider the following equivalent definition. Let $H$ be a simply connected real nilpotent Lie group $H$ admitting a lattice $\Gamma$ and with Lie algebra $\lh$. Then the quotient $M = \Gamma \backslash H$ is a compact real nilmanifold. Given an almost complex structure $J\colon \lh \to \lh$ we have a decomposition of the complexified Lie algebra $\lh_\IC = \lh^{1,0} \oplus \lh^{0,1}$, where $\lh^{1,0}$ is the $i$-eigenspace of the $\IC$-linear extension of $J$ and $\lh^{0,1}=\overline{\lh^{1,0}}$ is the $(-i)$-eigenspace. Extending this almost complex structure by left-translations we get an almost complex structure on $H$ descending to $M$.
This almost complex structure is integrable if and only if the Lie bracket is $J$-linear or equivalently if $[\lh^{1,0},\lh^{0,1}]=0$.
If integrability holds, $(H,J)$ is a complex Lie group and $(M,J)$ is a complex parallelisable nilmanifold.
Moreover, the canonical projection
\[\pi\colon(\lh,J)\to \lh^{1,0}, \qquad z\mapsto \tfrac12(z-iJz)\]
is an isomorphism of complex Lie algebras. Therefore,
defining the Lie bracket on $\lgbar$ by $[\overline x, \overline y]=\overline{[x,y]}$,
we have an identification
\[ \lg_\IC=\lh_\IC=\lg\oplus \lgbar, \]
where $[\lg, \lgbar]=[\lh^{1,0}, \lh^{0,1}]=0$.

\subsection{Deformation theory}

In general the deformation theory of a compact complex manifold $X$ is governed by the Kodaira--Spencer DGLA $(\KS_X, \delbar, [\cdot, \cdot])$. Here $\KS_X =\bigoplus_q \KS_X^q$, where $\KS_X^q$ is the vector space of $(0,q)$-forms with values in the holomorphic tangent bundle and $[\cdot, \cdot]$ is the Schouten bracket (see \cite[Section~8.3]{manetti2022lie} for details).  In the case of complex parallelisable nilmanifolds we can restrict our considerations to a finite-dimensional DGLA \cite{rollenske2011kuranishi}, which we describe in the following.

Let $X= \Gamma \backslash G$ be a complex parallelisable nilmanifold with Lie algebra $\lg$. Then we can identify elements in $\Wedge^q\lgbar^*$ with left-invariant differential forms of type $(0,q)$ on $X$ and the restriction of the usual differential $d=\del+\delbar$ to $\Wedge^q\lgbar^*$ can be expressed in terms of the Lie bracket on $\lg$ as follows: let $\overline\alpha \in \lgbar^*$ be a left-invariant $(0,1)$-form and let $x,y\in \lg_\IC$ be left-invariant vector fields on $X$. Then we have 
\[
	d\overline\alpha(x,y)=x(\overline\alpha(y))-y(\overline\alpha(x))-\overline\alpha([x,y])=-\overline\alpha([x,y]).
\]
Since $[\lg, \lgbar]=0$, we get $\delbar \overline\alpha = d\overline\alpha$ and every element of $\lg$ gives rise to a holomorphic vector field on $X$. The cohomology of the tangent sheaf $\Theta_X$ can now be computed via the complex
\[
0\to \lg \overset{0}{\to} \lgbar^*\tensor \lg\overset{\delbar}{\to}\Wedge^{2} \lgbar^*\tensor \lg\overset{\delbar}{\to}\dots
\]
obtained from $\left(\Wedge^\bullet\lgbar^*,\delbar\right)$ by tensoring with $\lg$.
In particular, we have
\[
H^1(X,\Theta_X) = {\Ann([\lgbar,\lgbar])}\tensor \lg.
\]
 Restricting the usual Schouten bracket on $X$ to left-invariant forms we get
\begin{equation}\label{eq:schouten}
[\overline\alpha \otimes x, \overline\beta \otimes y] = (\overline\alpha \wedge \overline\beta) \otimes [x,y]
\end{equation}
for $\overline\alpha \otimes x\in \Wedge^{p}\lgbar^*\otimes \lg$ and $\overline\beta \otimes y\in \Wedge^q\lgbar^*\otimes \lg$. This makes $(\Wedge^\bullet\lgbar^* \otimes \lg, \delbar, [\cdot, \cdot])$ into a DGLA.
By \cite{rollenske2011kuranishi}, we have the following:

\begin{prop}
	Let $X= \Gamma\backslash G$ be a complex parallelisable nilmanifold with Lie algebra $\lg$. Then the inclusion of $(\Wedge^\bullet\lgbar^* \otimes \lg, \delbar, [\cdot, \cdot])$ into the Kodaira--Spencer DGLA is a quasi-isomorphism. In particular, sufficiently small deformations of $X$ carry a left-invariant complex structure.
\end{prop}

\begin{rem}
	Under further assumptions a similar result holds true for general nilmanifolds with left-invariant complex structure  \cite{rollenske09}.
\end{rem}

Kuranishi proved in \cite{kuranishi62} that for every compact complex manifold there exists a versal family of deformations over a complex space known as the Kuranishi space. After choosing a hermitian metric, this space can be explicitly constructed in the space of vector-valued harmonic $(0,1)$-forms.
We briefly recall his construction for complex parallelisable nilmanifolds. A detailed exposition is given in \cite{rollenske2011kuranishi}.

The left-invariant complex structure on $X$ is uniquely determined by the eigenspace decomposition $\lg_\IC=\lg\oplus \lgbar$.
Every sufficiently small deformation of this complex structure can be described by a map $\Phi\colon\lgbar \to \lg$ such that the graph $(\id+\Phi) \lgbar$ of $\Phi$ in $\lg_\IC$ is the new space of $(0,1)$-vectors,
which uniquely determines an almost complex structure on $X$.
The new almost complex structure is integrable if and only if $(\id+\Phi) \lgbar$ is closed under the Lie bracket.
If we consider $\Phi$ as an element of $\lgbar^* \otimes \lg$, this condition can be equivalently described by the Maurer--Cartan equation
\begin{equation}\label{eq: MC}
	\delbar \Phi +\tfrac12[\Phi, \Phi]=0.
\end{equation}
The space $H^1(X,\Theta_X)$ parametrises infinitesimal deformations, which correspond to first-order solutions of the Maurer--Cartan equation \eqref{eq: MC}. The obstruction to integrating an infinitesimal deformation is given as follows:
we consider the power series expansion
\begin{equation}\label{eq: series}
\Phi(t) = \sum_{k\geq 1}\Phi_k(t)
\end{equation}
of $\Phi$ around $0$, where $n=\dim H^1(X,\Theta_X)$ and $t=(t_1,\dots,t_n) \in \IC^n$.
Then each $\Phi_k$ is a homogeneous polynomial of degree~$k$ in the variables~$t_i$ and \eqref{eq: MC} is equivalent to the system of equations
\begin{equation}\label{eq:Phi_k}
 \delbar\Phi_k(t) = -\frac{1}{2}\sum_{0<i<k}[\Phi_i(t), \Phi_{k-i}(t)], \qquad k\geq 1.
\end{equation}
For $k=1$ this amounts to $\Phi_1(t)$ being $\delbar$-closed. One says that $X$ has \emph{unobstructed} deformations if there are representatives $\mu_1,\dots,\mu_n$ for a basis of $H^1(X,\Theta_X)$ such that the system of equations \eqref{eq:Phi_k} has a solution for every choice of parameters $t_1,\dots,t_n$ defining $\Phi_1(t) = \sum_{i=1}^nt_i\mu_i$. For some examples of complex parallelisable nilmanifolds, solutions to \eqref{eq:Phi_k} were explicitly computed in \cite{rollenske2011kuranishi} and \cite{popovici2022deformations}.

After choosing a left-invariant hermitian metric on $X$ one can single out a distinguished left-invariant solution by recursively defining
\begin{equation}\label{eq: Green}
	\Phi_k(t) = -\delbar^*G\left(\frac{1}{2}\sum_{0<i<k}[\Phi_i(t), \Phi_{k-i}(t)]\right).
\end{equation}
Here $G$ denotes the Green operator, which is inverse to the Laplacian restricted to the orthogonal complement of the space of harmonic forms. Choosing a basis of harmonic representatives $\mu_1, \dots, \mu_n$ one defines the so-called obstruction map for $\mu = \sum_{i=1}^n t_i\mu_i$ as $\obs(\mu) = H[\Phi(t), \Phi(t)]$, where $\Phi(t)$ is constructed as in \eqref{eq: Green} and $H$ is the projection to the space of harmonic $(0,2)$-forms with values in the holomorphic tangent bundle.
Kuranishi's theorem can now be stated as follows.

\begin{theo}[Kuranishi]
	There is a versal family of deformations of $X$ over the Kuranishi space 
	\[\Kur(X):=\{\mu\in \kh^1(\Theta_X)\mid \|\mu\|<\epsilon; \obs(\mu)=0\}\]
	for sufficiently small $\epsilon>0$, where $\kh^1(\Theta_X)$ is the space of harmonic 1-forms with values in $\Theta_X$.
\end{theo}

It was proven in \cite[Theorem~4.5]{rollenske2011kuranishi} that $\Phi_k(t)$ vanishes for $k>\nu$ if $\lg$ is $\nu$-step nilpotent, which guarantees the convergence of the formal power series $\Phi(t)$ for all parameters $t_1,\ldots,t_n$.
Therefore, the Kuranishi space of $X$ is described by polynomial equations of degree at most $\nu$.

Since the definition of $\Kur(X)$ depends on the choice of a hermitian metric, the described construction is not canonical. However, for different choices of hermitian structures the germs of the resulting complex spaces are (non-canonically) isomorphic. Furthermore, the complex manifold $X$ has unobstructed deformations if and only if the Kuranishi space of $X$ is smooth, that is, $\Kur(X)$ is a complex manifold.

\begin{exam}
	Possibly the most prominent example of a complex parallelisable nilmanifold which is not a complex torus is the so-called Iwasawa manifold $I$. It can be described as a quotient of the $3$-dimensional complex Heisenberg group and thus its associated Lie algebra is the freely nilpotent Lie algebra $\frn_{2,2}$. In \cite{nakamura75} Nakamura showed that $I$ has unobstructed deformations. Moreover, he showed that there are small deformations of $I$ which are no longer complex parallelisable.
\end{exam}

\begin{rem}
	It was proven in \cite[Corollary~5.2]{rollenske2011kuranishi} that the phenomenon of small deformations which are no longer complex parallelisable occurs for every complex parallelisable nilmanifold which is not a complex torus. By a result of \cite{wavrik69} this implies that the Kuranishi family of a complex parallelisable nilmanifold $X$ is never universal unless $X$ is a complex torus.
	
	By \cite[Theorem~2.4]{popovici2022deformations}, the isomorphism type of the complex Lie algebra $\lg$ does not change in a small complex parallelisable deformation of $X$. Therefore, one could treat these also by deforming the lattice $\Gamma$ inside a fixed complex Lie group $G$, as is usually done for complex tori.
\end{rem}

\subsection{Unobstructed complex parallelisable nilmanifolds}

In this section we prove our main theorem and as a consequence obtain a partial classification of complex parallelisable nilmanifolds with unobstructed deformations.

The key observation is the following:

\begin{lem}\label{lem:mce}
	A vector-valued $1$-form $\Phi\in\lgbar^*\otimes\lg$ satisfies the Maurer--Cartan equation \eqref{eq: MC}
	if and only if the linear map $\Phi\colon\lgbar\to\lg$ is a Lie algebra homomorphism.
\end{lem}
\begin{proof}
	Evaluating the bracket \eqref{eq:schouten} on $\overline x,\overline y\in\lgbar$ we get
	\[ \tfrac12[\Phi,\Phi](\overline x,\overline y) = \tfrac12[\Phi(\overline x),\Phi(\overline y)]-\tfrac12[\Phi(\overline y),\Phi(\overline x)]=[\Phi(\overline x),\Phi(\overline y)]. \]
	Suppose $\Phi = \overline\alpha \otimes z$. Then we have
	\begin{align*}
		\delbar\Phi(\overline x, \overline y) &= (\delbar\overline\alpha \otimes z - \overline\alpha \wedge \delbar z)(\overline x,\overline y)\\
		&= d\overline\alpha(\overline x,\overline y) z\\
		&= -\overline\alpha([\overline x,\overline y]) z \\
		&= -\Phi([\overline x,\overline y]).
	\end{align*}
	The second equality comes from the fact that $[\lg,\lgbar] =0$, which implies both $\delbar z=0$ and $\partial \overline \alpha =0$.
	Therefore, the Maurer--Cartan equation states \[ \Phi([\overline x,\overline y])=[\Phi(\overline x),\Phi(\overline y)], \] which is the definition of a Lie algebra homomorphism.
\end{proof}

We are now ready to prove Theorem~\ref{thm: main deformation}.
\begin{thm}\label{thm:main-again}
Let $X$ be a compact complex parallelisable nilmanifold with associated nilpotent complex Lie algebra $\lg$. Then $X$ has unobstructed deformations if and only if $\lg\cong\lgbar$ and $\lg$ is pseudo-free.
\end{thm}
\begin{proof}
The manifold $X$ has unobstructed deformations if and only if for all $\mu \in H^1(X,\Theta_X) = {\Ann([\lgbar,\lgbar])}\tensor \lg $ there exists a formal power series $\Phi$ as in \eqref{eq: series} satisfying the Maurer--Cartan equation and such that $\Phi_1 = \mu$. Given the existence of such a power series one can always construct a new solution $\Phi = \sum_{k\geq 1} \Phi_k \in \lgbar^* \otimes \lg$, as described in \eqref{eq: Green}, such that each $\Phi_k$ has no $\delbar$-closed summands for $k\geq 2$, i.\,e., no summands are contained in $ {\Ann([\lgbar,\lgbar])}\tensor \lg $.

After choosing a hermitian metric, we may identify the quotient $\lg/[\lg,\lg]$ with the subspace $V=\left<x_1,\ldots,x_m\right>\subset\lg$ spanned by an orthonormal basis $x_1,\ldots,x_m$. Then the $\Phi_1$-part of $\Phi\in\lgbar\otimes\lg$, regarded as a linear map $\lgbar\to\lg$, is simply given by its restriction to $\overline V=\left<\overline{x_1},\ldots,\overline{x_m}\right>$.

Therefore, if we regard $\mu$ as a linear map from $\overline V$ to $\lg$, the infinitesimal deformation $\mu$ is unobstructed if and only if there exists a linear map $\Phi\colon \lgbar \to \lg$ extending $\mu$ and satisfying the Maurer--Cartan equation.
By Lemma~\ref{lem:mce}, this is equivalent to $\Phi\colon\lgbar\to\lg$ being a Lie algebra homomorphism.
Hence, $X$ has unobstructed deformations if and only if every linear map $\overline V\to\lg$ extends to a Lie algebra homomorphism $\lgbar\to\lg$.

We claim that the latter property is satisfied if and only if $\lg\cong\lgbar$ and $\lg$ is pseudo-free. By Lemma~\ref{lem: pseudo-free}, we know that $\lg$ is pseudo-free if and only if every linear map $V\to\lg$ extends to a Lie algebra endomorphism $\lg\to\lg$.

First assume that $\lg\cong\lgbar$ and $\lg$ is pseudo-free.
By Lemma~\ref{lem:pseudo-free isomorphic}, there exists an isomorphism $g\colon\lgbar\to\lg$ such that $g(\overline{x_i})=x_i$.
Now let $\varphi\colon\overline V\to\lg$ be a given linear map.
We can extend the linear map $V\to\lg$ defined by $x_i\mapsto\varphi(\overline{x_i})$ to a Lie algebra endomorphism $\lg\to\lg$.
Precomposing this endomorphism with $g\colon\lgbar\to\lg$, we obtain a Lie algebra homomorphism $\lgbar\to\lg$ extending $\varphi$.

Conversely, suppose that every linear map $\overline V\to\lg$ extends to a Lie algebra homomorphism $\lgbar\to\lg$.
Applying this to the linear map given by $\overline{x_i}\mapsto x_i$, we obtain a Lie algebra homomorphism $h\colon\lgbar\to\lg$
such that $h(\overline{x_i})=x_i$. Since $[\overline x,\overline y]=\overline{[x,y]}$, the map $h'(x)=\overline{h(\overline x)}$
defines a ($\IC$-linear) Lie algebra homomorphism $\lg\to\lgbar$. Clearly, $h'$ is inverse to $h$ since $h'(x_i)=\overline{x_i}$.
This proves that $\lg\cong\lgbar$ as complex Lie algebras.
It remains to show that $\lg$ is pseudo-free. Let $\varphi\colon V\to\lg$ be a linear map.
We can extend the linear map $\overline V\to\lg$ defined by $\overline{x_i}\mapsto\varphi(x_i)$ to a Lie algebra homomorphism $\lgbar\to\lg$.
Precomposing this homomorphism with $h'\colon\lg\to\lgbar$, we obtain a Lie algebra endomorphism $\lg\to\lg$ extending $\varphi$.
\end{proof}

\begin{exam}
We want to illustrate the necessity of the reality condition in the above theorem.
Suppose that $X$ has the Lie algebra $\lg_\mu$ from Example~\ref{exam: smallest}.
If $\mu\notin\IP^1_\IR$, then $\lg_\mu$ is still pseudo-free, although $X$ does not have unobstructed deformations.
Concretely, the infinitesimal deformation $\Phi_1\colon\overline V\to\lg$ given by $\overline x_1\mapsto x_1$
and $\overline x_2\mapsto x_2$ does not extend to a homomorphism $\Phi\colon\overline\lg\to\lg$.
Indeed, if this would be the case, then the verbal ideal defining $\lg_\mu$ would need to contain
$\mu_1v_1+\mu_2v_2$ and $\overline{\mu_1}v_1+\overline{\mu_2}v_2$. However, since $\mu\notin\IP^1_\IR$, this would imply
that it contains both highest weight vectors $v_1$ and $v_2$. Hence, the irreducible representation $V_{(5,2)}$ would occur with multiplicity~$2$ in the ideal defining $\lg_\mu$,
which is a contradiction. Therefore, we have $\lg_\mu\not\cong\overline{\lg_\mu}$ by Lemma~\ref{lem:pseudo-free isomorphic}.

Choosing for example $\mu=(1:i)$ it is clear that the underlying real Lie algebra is defined over $\IQ$, so \cite{malcev51} guarantees the existence of a lattice in the associated Lie group. The quotient is a compact complex parallelisable nilmanifold with pseudo-free Lie algebra but obstructed deformations.
\end{exam}

\begin{cor}\label{cor:nilmanifolds20}
	Let $X$ be a complex parallelisable nilmanifold with unobstructed deformations which is not a complex torus. If $\dim X \leq 20$ then the associated Lie algebra is one of the Lie algebras in Table~\ref{tab: small dim}. All of the $19$ individual Lie algebras can occur and Lie algebras in the family $\{\lg_\mu\}$ can occur at least for $\mu \in \IP^1_{\IQ}$.
	
	Moreover, the family $\{\lg_\mu\}$ provides infinitely many complex homotopy types of complex parallelisable nilmanifolds with unobstructed deformations.
\end{cor}
\begin{proof}
	 Except for the Lie algebras in the family $\{\lg_\mu\}$ with $\mu\notin\IP^1_\IR$, all Lie algebras in Table~\ref{tab: small dim} satisfy the reality condition $\lg\cong\lgbar$.
     If we further assume $\mu \in \IP^1_{\IQ}$, then Mal'cev's criterion \cite{malcev51} proves the existence of a lattice in the associated complex Lie group for all Lie algebras in Table~\ref{tab: small dim}.
	 
	 As explained in \cite[Section~4]{latorre2018family} two manifolds have the same complex homotopy type if and only if their $\IC$-minimal models are isomorphic. In the case of a real nilmanifold $\Gamma\backslash H$ with Lie algebra $\lh$ Nomizu's theorem \cite{nomizu54} implies that the $\IC$-minimal model is given by the Chevalley complex $(\Wedge^{\bullet}\lh_\IC^*, d)$. Since the differential $\lh_\IC^* \to \Wedge^2 \lh_\IC^*$ encodes the Lie bracket by Cartan's formula, non-isomorphic Lie algebras have non-isomorphic minimal models.
	 	 This concludes the proof by Lemma~\ref{lem:pseudo-free isomorphic}. 
\end{proof}

Finally, we are in a position to  prove our classification result.

\begin{proof}[Proof of Corollary~\ref{cor: classification cx parall}]
The first two statements follow from Theorem~\ref{thm:main-again} and Theorem~\ref{thm: pseudofree classification},
combined with the observation that the occuring Lie algebras $\gothn_{m,\nu}$ and the ones listed in Table~\ref{tab: nilpotency index 5}
satisfy the reality condition $\lg\cong\lgbar$. Furthermore, these Lie algebras are geometrically realizable, as the underlying real Lie algebras are defined over $\IQ$, and hence Mal'cev's criterion \cite{malcev51} guarantees the existence of a lattice.
The last two statements in Corollary~\ref{cor: classification cx parall} were proven in Corollary~\ref{cor:nilmanifolds20}.
\end{proof}

\bibliographystyle{halpha}
\bibliography{references}

\end{document}